  \documentclass[11pt,a4paper,twoside]{article}
  \usepackage[margin=3.13cm]{geometry}

  \usepackage{amsmath,amssymb,amsthm,amsfonts}
  \usepackage{mathabx}
    \usepackage{xspace,xcolor}
  \usepackage{parskip,fancyhdr,url}
  
  \usepackage{fourier}
  \usepackage{enumitem}
    \usepackage{bm}
  \usepackage{todonotes}
  \usepackage[
     breaklinks,colorlinks,
     citecolor=blue,linkcolor=blue,urlcolor=teal,
     pdftitle={Epimorphism classes and relatively maximal metrics in large-scale geometry},pdfauthor={Robert Tang}
  ]{hyperref}
  \usepackage{tikz, tikz-cd}
  \usetikzlibrary{
    arrows, calc, matrix, decorations.pathmorphing, shapes, arrows
    }
  \usepackage{mathtools}
  \usepackage{centernot}

  \usetikzlibrary{decorations.markings}
\tikzset{negated/.style={
        decoration={markings,
            mark= at position 0.5 with {
                \node[transform shape] (tempnode) {$\backslash$};
            }
        },
        postaction={decorate}
    }
}

  \setenumerate[0]{label=(\arabic*)}

  \makeatletter
    \def\tagform@#1{\maketag@@@{%
     \textbf{(\ignorespaces#1\unskip\@@italiccorr)}}}%
     \renewcommand{\eqref}[1]{\textup{\maketag@@@{(\ignorespaces%
          {\ref{#1}}\unskip\@@italiccorr)}}}
  \makeatother

  \newcommand\address[1]{}
  \newcommand\email[1]{}
  \newcommand\dedicatory[1]{}

  \theoremstyle{plain}
  \newtheorem{theorem}[equation]{Theorem}
  \newtheorem{proposition}[equation]{Proposition}
  \newtheorem{corollary}[equation]{Corollary}
  \newtheorem{lemma}[equation]{Lemma}

  \theoremstyle{definition}
  \newtheorem{definition}[equation]{Definition}
  \newtheorem{remark}[equation]{Remark}
  \newtheorem{example}[equation]{Example}

  \numberwithin{equation}{section}

  \newcommand{\bsf}[1]{\ensuremath{\bm{\mathsf{#1}}}\xspace}

  \newcommand{\cat}{\ensuremath{\bsf{C}\xspace}}

  \newcommand{\CLip}{\ensuremath{\bsf{CLip}\xspace}}

  \newcommand{\Born}{\ensuremath{\bsf{Coarse}\xspace}}

  \newcommand{\barcat}{\overline{\cat}}
  \newcommand{\barCL}{\overline{\CLip}}
  \newcommand{\barBorn}{\overline{\Born}}

  \DeclareMathOperator{\Eq}{Eq}
  \DeclareMathOperator{\Coeq}{Coeq}

  \DeclareMathOperator{\Rips}{Rips}

  \newcommand{\set}[1]{\ensuremath{\left\{ {#1} \right\}}\xspace}

  \newcommand{\st}{\ensuremath{\,\, \colon \,\,}\xspace}
  \newcommand{\from}{\ensuremath{\colon \thinspace}\xspace}

  \newcommand{\Z}{\ensuremath{\mathbb{Z}}\xspace}
  \newcommand{\N}{\ensuremath{\mathbb{N}}\xspace}
  \newcommand{\R}{\ensuremath{\mathbb{R}}\xspace}

  \newcommand{\Hy}{\ensuremath{\mathbb{H}}\xspace}

  \newcommand{\proofof}[1]{\hfill\newline\noindent\emph{Proof of {#1}.} }

  \newcommand{\param}{{\mathchoice{\mkern1mu\mbox{\raise2.2pt\hbox{$
  \centerdot$}}
  \mkern1mu}{\mkern1mu\mbox{\raise2.2pt\hbox{$\centerdot$}}\mkern1mu}{
  \mkern1.5mu\centerdot\mkern1.5mu}{\mkern1.5mu\centerdot\mkern1.5mu}}}

  \fancypagestyle{plain}

   \AtEndDocument{\bigskip{\footnotesize%
   Department of Pure Mathematics, Xi'an Jiaotong--Liverpool University, 111 Ren'ai Road, Suzhou Industrial Park, Suzhou, Jiangsu, 215123, P. R.~China
   \par \texttt{robert.tang@xjtlu.edu.cn}}}

\begin{document}

      \title{Epimorphism classes and relatively maximal metrics in large-scale geometry}
    \author{Robert Tang} \date{15 December 2025}
  \maketitle \thispagestyle{empty}

  \begin{abstract}

  We consider epimorphisms and several variant notions -- split, effective, regular, strong, and extremal -- and determine which of these coincide in the metric coarse and coarsely Lipschitz categories. In particular, we characterise extremal epis in the coarsely Lipschitz category via a relative maximality condition on the codomain metric; this can be viewed as a morphism-relative analogue of Rosendal's maximal metrics for topological groups.

  \end{abstract}


  \section{Introduction}

  In this note, we study several classes of epimorphisms from the viewpoint of large-scale geometry. We first consider the metric coarse category $\barBorn$, whose objects are extended metric spaces and whose morphisms are closeness classes of controlled maps. In $\barBorn$, the epimorphisms are the coarsely surjective maps and the split epimorphisms are the coarse retractions. Our first main result records which intermediate epi-classes  coincide. Here, the straight implication arrows hold in any category; we establish the curved arrows.

  \begin{theorem}\label{thm:epi-coarse}
   The following implications hold in $\barBorn$.
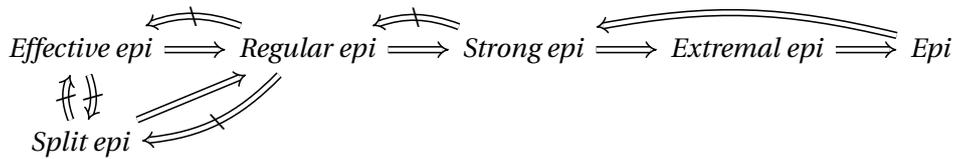
   \begin{center}
   \begin{tikzcd}
	{\textrm{Effective epi}} & {\textrm{Regular epi}} & {\textrm{Strong epi}} & {\textrm{Extremal epi}} & {\textrm{Epi}}\\
	 {\textrm{Split epi}}
	\arrow[from=2-1, to=1-2, Rightarrow]
	\arrow[from=2-1, to=1-1, Rightarrow, negated, bend left = 20]
	\arrow[from=1-1, to=2-1, Rightarrow, negated, bend left = 20]
	\arrow[from=1-1, to=1-2, Rightarrow]
	\arrow[from=1-2, to=1-3, Rightarrow]
	\arrow[from=1-3, to=1-2, Rightarrow, negated, bend right = 20]
	\arrow[from=1-2, to=1-1, Rightarrow, negated, bend right = 20]
	\arrow[from=1-2, to=2-1, Rightarrow, negated, bend left = 20]
	\arrow[from=1-3, to=1-4, Rightarrow]
	\arrow[from=1-4, to=1-5, Rightarrow]
	\arrow[from=1-5, to=1-3, Rightarrow, bend right = 12, shift right = 1]
    \end{tikzcd}
    \end{center}
  \end{theorem}

  Informally, a regular epi behaves like a quotient, while an effective epi is a quotient by a kernel. In $\barBorn$, we show regular epis are those maps whose codomain, up to coarse equivalence, is obtained by ``coarsely gluing'' the domain to itself along some binary relation. Effective epis are a special case of these where the gluing relations coarsely stabilise to a canonical ``coarse kernel'' at some threshold; for example, surjective group homomorphisms with finitely generated image and kernel (Lemma \ref{lem:eff-group}). We characterise these via an explicit stability criterion for the associated coarse quotient and coarse kernel filtrations.

  Next, we turn to the coarsely Lipschitz category $\barCL$, where morphisms are required to have affine upper control. This is the natural setting for studying finitely generated groups equipped with word metrics (unique up to quasi-isometry) and is central to geometric group theory.

  Many descriptions of epi classes from $\barBorn$ carry over to $\barCL$ with quasi-isometry replacing coarse equivalence, but some distinctions emerge.
  In this category, extremal epis are geometrically meaningful: a coarsely Lipschitz map $f \from (X,d_X) \to (Y,d_Y)$ represents an extremal epi if and only if it is coarsely surjective and $d_Y$ is affinely maximal among all coarsely equivalent metrics $d'$ on $Y$ such that the induced map $(X,d_X) \to (Y,d')$ is coarsely Lipschitz (Proposition \ref{prop:ext-reg}). This maximality condition can be viewed as a morphism-relative analogue of Rosendal’s notion of maximal metrics on topological groups \cite{Ros22, Ros23}. Further, we show that an epi in $\barCL$ can be promoted (by replacing the codomain metric with a coarsely equivalent one) to an extremal epi in $\barCL$ precisely when the given morphism represents a regular epi in $\barBorn$; this can be done explicitly via an augmented Rips graph construction (Proposition \ref{prop:promote}). This yields a canonical quasi-isometry class of metrics on the codomain, mirroring Rosendal's results for maximal metrics on Polish groups.

  Our principal structural result is that extremal epis are regular in $\barCL$. In other words, relative maximality of the codomain metric is equivalent to quotient-like behaviour of the map. Since regular epis are stable under pushouts (when they exist) \cite[Proposition 11.18]{AHS06}, it follows that relative maximality persists under coarse gluings in the coarsely Lipschitz setting.

  \begin{theorem}\label{thm:epi-clip}
  The following implications hold in $\barCL$.
   \begin{center}
   \begin{tikzcd}
	{\textrm{Effective epi}} & {\textrm{Regular epi}} & {\textrm{Strong epi}} & {\textrm{Extremal epi}} & {\textrm{Epi}}\\
	 {\textrm{Split epi}}
	\arrow[from=2-1, to=1-2, Rightarrow]
	\arrow[from=2-1, to=1-1, Rightarrow, negated, bend left = 20]
	\arrow[from=1-1, to=2-1, Rightarrow, negated, bend left = 20]
	\arrow[from=1-1, to=1-2, Rightarrow]
	\arrow[from=1-2, to=1-3, Rightarrow]
	\arrow[from=1-5, to=1-4, Rightarrow, negated, bend right = 30]
	\arrow[from=1-2, to=1-1, Rightarrow, negated, bend right = 20]
	\arrow[from=1-2, to=2-1, Rightarrow, negated, bend left = 20]
	\arrow[from=1-3, to=1-4, Rightarrow]
	\arrow[from=1-4, to=1-5, Rightarrow]
	\arrow[from=1-4, to=1-2, Rightarrow, bend right = 12]
    \end{tikzcd}
\end{center}
  \end{theorem}

   It is worth contrasting Theorems \ref{thm:epi-coarse} and \ref{thm:epi-clip} with their dual counterparts. In prior work \cite{Tang-mono}, it was shown that $\barBorn$ and $\barCL$ are coregular categories; this implies that the effective, regular, strong, and extremal monos coincide. Moreover, the regular monos in $\barBorn$ (resp.~$\barCL$) are identified with the coarse (resp.~quasi-isometric) embeddings.

    \begin{theorem}[{\cite[Corollary 1.4, Proposition 1.5, Theorem 1.8]{Tang-mono}}]\label{thm:regmono}
       In $\barBorn$ (or $\barCL$), the effective, regular, strong, and extremal monos coincide, while the split monos form a strictly stronger class. Every mono in $\barBorn$ is regular; while the same does not hold in $\barCL$. \qed
  \end{theorem}

      In summary, our results demonstrate that coarse analogues of quotients and kernels depend more on categorical context than do coarse analogues of embeddings.

  \textbf{Relation to previous work.}
  The notion of a \emph{coarse quotient} between metric spaces was introduced by Zhang \cite{Zha15} and has been subsequently developed in \cite{Zha18, Bra22}. We show that these coarse quotients are effective epis in $\barBorn$ (Example \ref{ex:zhang}). Higginbotham--Weighill \cite{HW19} introduced a related notion, known as \emph{weak coarse quotients}, in the setting of large-scale structures and showed these characterise regular epimorphisms in $\barBorn$; they also observed that every weak coarse quotient of a large-scale structure admits a metric realisation. Our treatment provides a purely metric approach to characterising the regular epis.

  Rosendal's maximal metrics framework has stimulated much recent progress in the geometric group theory of topological groups \cite{MRo18, AV20, MRa23}, especially beyond the locally compactly generated setting. It is natural to ask whether maximal metrics on topological groups admits a purely categorical formulation. Our results suggest an approach via extremal epis but this is not immediate: topological group quotients require continuity, while here quotients are coarse.  Nevertheless, it could be interesting to explore relative maximality for topological group homomorphisms within an appropriate categorical framework.

  \subsection*{Acknowledgements}
The author is supported by the Gusu Innovation and Entrepreneurship Leading Talents Programme (ZXL2022473) under the Suzhou Science and Technology Bureau; the National Natural Science Foundation of China (NSFC 12101503); and the XJTLU Research Development Fund (RDF-23-01-121).

  \section{Preliminaries}

    \subsection{Epimorphism classes}

        In any category, an \emph{equaliser} of a parallel pair $f,g \from A \rightrightarrows B$ is a morphism $e \from E \to A$ such that $fe = ge$ with the universal property: for every morphism $h \from X \to A$ with $fh = gh$, there exists a unique morphism $h' \from X \to E$ such that $eh' = h$. Coequalisers are defined dually. Any (co)equaliser, if it exists, is unique up to unique isomorphism. The \emph{kernel pair} of a morphism $f \from X \to Y$ is the pullback of $f,f \from X \rightrightarrows Y$ (if it exists). Let us recall some useful constructions of coequalisers and kernel pairs.

    \begin{proposition}[{\cite[Proposition 11.14]{AHS06}}]\label{prop:coeq-pushout}
    In a category with binary coproducts, if
     \begin{center}
   \begin{tikzcd}
     A \sqcup B \arrow[r, "f \sqcup 1_Y"] \arrow[d,"g \sqcup 1_Y"'] & B \arrow[d,"q"] \\
     B \arrow[r,"q"'] & Y
    \end{tikzcd}
    \end{center}
    is a pushout then $q \from B \to Y$ is the coequaliser of $f,g \from A \rightrightarrows B$. \qed
    \end{proposition}

    \begin{proposition}\label{prop:kernel-eq}
    In a category with binary products, the kernel pair of $f \from X \to Y$ is given by ${k_1,k_2 \from K \rightrightarrows X}$ if and only if $k_1 \times k_2 \from K \to X_1 \times X_2$ is the equaliser of $f\pi_1, f\pi_2 \from X\times X \rightrightarrows Y$.
    \begin{center}
    \begin{tikzcd}
        K \arrow[dr, "k_1 \times k_2"] \arrow[drr, "k_1", bend left = 20] \arrow[ddr, "k_2"', bend right = 20] & & \\
        & X \times X \arrow[r, "\pi_1"'] \arrow[d,"\pi_2"] & X \arrow[d,"f"] \\
        & X \arrow[r,"f"'] & Y
        \end{tikzcd}
        \end{center}
    \end{proposition}
    \proof
    The forwards implication is {\cite[Proposition 11.11]{AHS06}}. The converse is straightforward.
    \endproof

We now introduce the various epimorphism classes of interest.

\begin{definition}[Classes of epimorphisms]
An epimorphism $f \from X \to Y$ is:
\begin{enumerate}
  \item an \emph{extremal epi} if whenever $f = me$ with $m$ a mono, then $m$ is an isomorphism.
  \item a \emph{strong epi} if for every commutative square
   \begin{center}
   \begin{tikzcd}
     X \arrow[r] \arrow[d,"f"'] & A \arrow[d,"m"] \\
     Y \arrow[r] \arrow[ur,"d", dashed ]& B
    \end{tikzcd}
    \end{center}
  with $m$ a mono, there exists a unique $d\from Y\to A$ making both triangles commute,
  \item a \emph{regular epi} if it is the coequaliser of some pair of parallel arrows,
  \item an \emph{effective epi} if its kernel pair exists and $f$ is the coequaliser of that kernel pair, and
  \item a \emph{split} epi if it admits a right inverse.
\end{enumerate}
\end{definition}

    Extremal, strong, regular, effective, and split monos are defined dually.

    \begin{proposition}[{\cite[Proposition 4.3.6]{Bor94}}]\label{prop:implications}
     In any category, the following implications hold: extremal epi $\implies$ regular epi $\implies$ strong epi $\implies$ extremal epi $\implies$ epi. \qed
    \end{proposition}

    \begin{lemma}[{\cite[Proposition 7.59]{AHS06}}]\label{lem:split}
     Any split epi is a regular epi. \qed
    \end{lemma}

    \begin{lemma}[{\cite[Proposition 11.22]{AHS06}}]\label{lem:reg-kernel}
    Any regular epimorphism admitting a kernel pair is an effective epimorphism. \qed
    \end{lemma}

     \begin{lemma}\label{lem:strong}
      In any category, every mono is strong if and only if every epi is strong.
        \end{lemma}

  \proof
  Let $e \from A \to B$ be an epimorphism and $m \from X \to Y$ be a monomorphism. Suppose that $f \from A \to X$ and $g \from B \to Y$ are morphisms such that $mf = ge$. If $m$ is a strong mono, then there exists a unique morphism $d \from B \to X$ such that $f = de$ and $md = g$, hence $e$ is a strong epi. Similarly, if $e$ is a strong epi, then such $d$ exists, hence $m$ is a strong mono.
  \endproof

    \subsection{Coarse geometry}

    Background on coarse geometry can be found in \cite{Roe03, NY12, CH16, LV23}.

    Given an extended metric space $(X,d_X)$ and $\kappa \geq 0$, we shall write $x \approx x'$ to mean $d_X(x,x') \leq \kappa$ for $x,x' \in X$. Say two maps $f,g \from X \rightrightarrows Y$ are \emph{$\kappa$--close} if $fx \approx_\kappa gx$ for all $x \in X$; and \emph{close}, denoted $f \approx g$, if they are $\kappa$--close for some $\kappa \geq 0$. Closeness is an equivalence relation on maps between a given pair of metric spaces. Write $\bar f$ for the closeness class of $f$. Moreover, closeness is preserved under composition, that is, $\overline{ff'} = \bar f \bar f'$ whenever defined.

    An \emph{upper control} for $f\from X \to Y$ is a proper increasing function $\rho \from [0,\infty) \to [0,\infty)$ such that $d_Y(fx,fx') \leq \rho d_X(x,x')$ for all $x,x' \in X$ such that $d_X(x,x') < \infty$.    We say $f$ is \emph{controlled} if it admits an upper control, and \emph{coarsely Lipschitz} if it has an affine upper control. A proper increasing function $\theta \from [0,\infty) \to [0,\infty)$ is a \emph{lower control} if $\theta d_X(x,x') \leq d_Y(fx,fx')$ for all $x,x'\in X$, with the convention that $ d_X(x,x') = \infty \implies d_Y(fx,fx') = \infty$. The existence of a lower control is equivalent to $f$ being \emph{uniformly metrically proper}: for every $R \geq 0$ there exists $S \geq 0$ such that $d_Y(fx,fx') \leq R \implies d_X(x,x') \leq S$ for all $x,x'\in X$. We say $f$ is \emph{coarsely surjective} if there exists $r \geq 0$ such that $N_r(f(X)) = Y$, where $N_r$ denotes the (closed) metric $r$--neighbourhood. Call $f$ a \emph{coarse equivalence} if it is coarsely surjective and admits a lower control; in addition, if $f$ admits an affine lower control then it is called a \emph{quasi-isometry}.

    Let $\Born$ be the category whose objects are extended metric spaces, with controlled maps as morphisms. Let $\CLip \subset \Born$ be the wide subcategory whose morphisms are coarsely Lipschitz. The \emph{metric coarse category} $\barBorn$ and the \emph{coarsely Lipschitz category} $\barCL$ are the respective quotient categories with the same objects, and with closeness classes as morphisms. Since many of our arguments will work for either category, we shall write $\barcat$ to denote either $\barBorn$ or $\barCL$.

    We recall some characterisations of morphisms from \cite{CH16}; see \cite{Tang-mono} for proofs in the setting of extended metric spaces.

    \begin{proposition}[{\cite[Propositions 3.A.16, 3.A.22]{CH16}}]\label{prop:morphisms}
    A morphism $\bar f$ in $\barcat$ is a mono (resp.~epi) if and only if any (hence every) representative $f\in \bar f$ admits a lower control (resp.~is coarsely surjective). A morphism $\bar f$ is an isomorphism in $\barBorn$ (resp.~$\barCL$) if and only if any (hence every) representative $f\in \bar f$ is a coarse equivalence (resp.~quasi-isometry). \qed
    \end{proposition}

    \begin{lemma}[{\cite[Lemma 2.10]{Tang-rips}}]\label{lem:coarse-qi}
     Let $f \from (X,d_X) \to (Y,d_Y)$ be a coarse equivalence. Then there exists a metric $d'$ on $Y$ such that $f$ factors as $(X,d_X) \xrightarrow{f'} (Y,d') \xrightarrow{m} (Y,d_Y)$ where $f'$ is a quasi-isometry and $m$ is a coarse equivalence given by the underlying identity. \qed
    \end{lemma}

    The empty space is initial in $\barcat$. Given metric spaces $X,Y$, their binary product $X \times Y$ is realised by the Cartesian product equipped with the $\ell^\infty$--metric, together with the standard projections to each factor. The binary coproduct $X \sqcup Y$ is realised by the disjoint union equipped with the metric $d$ given by $d(x,x') = d_X(x,x')$, $d(y,y') = d_Y(y,y')$, and $d(x,y) = \infty$ for all $x,x' \in X$ and $y,y' \in Y$; together with the isometric embeddings $X \hookrightarrow X \sqcup Y$ and $Y \hookrightarrow X \sqcup Y$.

    \textbf{Coarse gluing.} Given maps $f \from A \to X$ and $g \from A \to Y$, the space $X \sqcup_A Y$ obtained by \emph{coarsely gluing} $X$ to $Y$ via $f$ and $g$ is defined as follows. Let $\Gamma$ be the weighted graph on vertex set $X \sqcup Y$ with two types of edges:
    \begin{itemize}
   \item \emph{internal edges} -- each pair of distinct $x,x' \in X$ (resp.~$y,y' \in Y'$) such that $d_X(x,x') < \infty$ (resp.~$d_Y(y,y') < \infty$) are joined by a edge of weight $d_X(x,x')$ (resp.~$d_Y(y,y')$),
   \item \emph{glued edges} -- for each $a \in A$, add an edge of weight 1 between $fa$ and $ga$.
  \end{itemize}
  Equip $\Gamma$ with the path metric, then define $X \sqcup_A Y$ to be its vertex set endowed with the induced metric. The inclusions yield 1--Lipschitz maps $X \to X \sqcup_A Y$ and $Y \to X \sqcup_A Y$.

  \begin{proposition}[{\cite[Proposition 3.3, Remark 3.5]{Tang-mono}}]\label{prop:pushout}
   Let $\bar f \from A \to X$ and $\bar g \from A \to Y$ be morphisms in $\barcat$. Choose representatives $f \in \bar f$ and $g \in \bar g$. Then
     \begin{center}
   \begin{tikzcd}
     A  \arrow[r, "\bar f"] \arrow[d,"\bar g"'] & X \arrow[d] \\
     Y \arrow[r] & X \sqcup_A Y
    \end{tikzcd}
    \end{center}
   is a pushout in $\barcat$. \qed
  \end{proposition}

    We show that coequalisers can also be computed via a coarse gluing construction. Given a pair of maps $f,g \from A \to X$, define $\Coeq(f,g)$ to be the space $X$ coarsely to glued to itself via $f$ and $g$; this follows the same construction as given above (taking $Y = \emptyset$), but with the glued edges connecting pairs of points in $X$. The underlying identity yields a 1--Lipschitz map $q \from X \to \Coeq(f,g)$. By construction, $qf \approx_1 qg$.
    
    \begin{proposition}[Coequalisers]
     Let $\bar f, \bar g \from A \rightrightarrows X$ be a parallel pair in $\barcat$, and choose representatives $f \in \bar f$, $g \in \bar g$. Then $\bar q \from X \to \Coeq(f,g)$ is the coequaliser of $\bar f, \bar g$ in $\barcat$.
    \end{proposition}

    \proof
    By Proposition \ref{prop:pushout}, the upper-left square in the commutative diagram
    \begin{center}
   \begin{tikzcd}
     A \sqcup X \arrow[r, "\bar f \sqcup \bar 1_X"] \arrow[d,"\bar g \sqcup \bar 1_X"'] & X \arrow[ddr,"\bar q", bend left = 10] \arrow[d]& \\
     X \arrow[rrd,"\bar q"', bend right = 10] \arrow[r] & X \sqcup_{A \sqcup X} X \arrow[dr, dashed, "\bar r"] & \\
     & & \Coeq(f,g)
    \end{tikzcd}
    \end{center}
    is a pushout in $\barcat$. Identify the underlying set of $X \sqcup_{A \sqcup X} X$ with $X \times \set{0,1}$; thus, there is a glued edge between $(x,0)$ and $(x,1)$ for each $x \in X$, and a glued edge between $(fa,0)$ and $(ga,1)$ for each $a \in A$. Define a map $r \from X \sqcup_{A \sqcup X} X \to \Coeq(f,g)$ coinciding with the underlying projection $X \times \set{0,1} \to X$. Any pair of points in $X \sqcup_{A \sqcup X} X$ connected by an edge of weight $w$ are mapped by $r$ either to a common point, or to a pair of points connected by an edge of weight $w$. Thus, $r$ is 1--Lipschitz. Now, $q$ factors as $X \to X \sqcup_{A \sqcup X} X \xrightarrow{r} \Coeq(f,g)$, where the first arrow is the inclusion to either copy of $X$. Thus $\bar r \from X \sqcup_{A \sqcup X} X \to \Coeq(f,g)$ is the canonical morphism in $\barcat$ granted by the universal property of the pushout.
    We claim that $\bar r$ is an isomorphism in $\barcat$. It will then follow that $\bar q, \bar q \from X \rightrightarrows \Coeq(f,g)$ yield a pushout, thence $\bar q$ is the desired coequaliser by Proposition \ref{prop:coeq-pushout}.
    
    Define $s \from \Coeq(f,g) \to X \sqcup_{A \sqcup X} X$ by $x \mapsto (x,0)$ on underlying sets. If $x,x'\in X$ are connected by an internal edge, then $sx, sx'$ are connected by an internal edge of the same weight; if there is a glued edge between $x,x'\in X$ then there is a path comprising two glued edges passing through the vertices $sx = (x,0), (x',1), (x',0) = sx'$. Therefore $s$ is 2--Lipschitz. Finally, $rs$ and $sr$ are close the identity maps on their respective domains, hence $\bar r$ is an isomorphism.
    \endproof

    \textbf{Coarse equaliser filtration.} Given maps $f,g \from X \to Y$ and $\kappa \geq 0$, define
    \[\Eq_\kappa(f,g) := \set{x \in X \st d_Y(fx, gx) \leq \kappa} \subseteq X\]
    equipped with the induced metric. These form a filtration $\Eq_*(f,g)$ with bonding maps given by the inclusions $\Eq_\sigma(f,g) \hookrightarrow \Eq_\tau(f,g)$ for $\sigma \leq \tau$. We say $\Eq_*(f,g)$ \emph{stabilises} in $\barcat$ if the bonding maps represent isomorphisms in $\barcat$ for all $\sigma \leq \tau$ large. Since the inclusions are isometric embeddings, stability is equivalent to the inclusions being coarsely surjective for $\sigma \leq \tau$ large; in particular, this does not depend on the choice of the ambient category $\barBorn$ or $\barCL$.
    
     \begin{proposition}[{\cite[Proposition 4.3]{Tang-mono}}]\label{prop:equaliser}
     Let $\bar f, \bar g \from X \rightrightarrows Y$ be a parallel pair in $\barcat$. Choose representatives $f \in \bar f, g \in \bar g$. Then $\bar f, \bar g$ admits an equaliser if and only if $\Eq_*(f,g)$ stabilises. In that case, $\Eq_\kappa(f,g) \hookrightarrow X$ realises the equaliser for all $\kappa \geq 0$ large. \qed
    \end{proposition}

  \section{Epimorphisms in the metric coarse category}

  In this section, we characterise epimorphism classes in $\barBorn$.
  \subsection{Strong and extremal epis}

  \begin{proposition}\label{prop:strong-coarse}
   Every epimorphism in $\barBorn$ is strong.
  \end{proposition}

  \proof
  This follows from Theorem \ref{thm:regmono} and Lemma \ref{lem:strong}.
  \endproof

  \subsection{Regular epis}

    Let $f \from X \to Y$ be a controlled map and $\pi_1, \pi_2 \from X \times X \rightrightarrows X$ denote the projections to each factor. We define two filtrations associated to $f$.

  \textbf{Coarse kernel filtration.}
  For $\sigma \geq 0$, let
  \[K_\sigma(f) := \Eq_\sigma(f\pi_1, f\pi_2) = \set{(x,x') \in X\times X \st d_Y(fx, fx') \leq \sigma}\subseteq X \times X\]
  equipped with the $\ell^\infty$--metric. These form a filtration $K_*(f)$ whose bonding maps are isometric embeddings $K_\sigma(f) \hookrightarrow K_\tau(f)$ for $\sigma \leq \tau$. Write $\iota_\sigma \from K_\sigma(f) \hookrightarrow X \times X$ for the inclusion.

  \textbf{Coarse quotient filtration.}
  For $\sigma \geq 0$, let $Q_\sigma(f) := \Coeq(\pi_1 \iota_\sigma, \pi_2 \iota_\sigma)$. These form a filtration $Q_*(f)$ with surjective 1--Lipschitz maps $Q_\sigma(f) \twoheadrightarrow Q_\tau(f)$ as bonding maps for $\sigma \leq \tau$ induced by the underlying identity on $X$. Write $q_\sigma \from X \to Q_\sigma(f)$ for the quotient map.
  
  We have the following diagram in $\barcat$ for each $\sigma \geq 0$, where the upper row is a coequaliser diagram. Up to closeness, the function $f_\sigma \from Q_\sigma(f) \to Y$ coinciding with $f$ on underlying sets is the unique map making the triangle commute. Therefore, $\bar f_\sigma$ is the canonical morphism in $\barcat$ granted by the universal property of the coequaliser.
    \begin{center}
    \begin{tikzcd}[row sep = normal, column sep = large]
    K_{\sigma}(f) \arrow[r, "\bar \pi_1 \bar\iota_\sigma", shift left] \arrow[r, "\bar \pi_2 \bar \iota_\sigma"', shift right] & X \arrow[dr, "\bar f"'] \arrow[r,"\bar q_\sigma"] & Q_{\sigma}(f) \arrow[d,"\bar f_\sigma", dashed]\\
    & & Y
    \end{tikzcd}
    \end{center}
%   Assume $f$ admits an upper control $\rho$. May assume coarsely subadditive.
  Let us consider stability of the filtrations $K_*(f)$ and $Q_*(f)$; that is, when all bonding maps are isomorphisms in $\barcat$ for $\sigma \leq \tau$ large. Stability of $K_*(f)$ is equivalent to $K_\sigma(f) \hookrightarrow K_\tau(f)$ being coarsely surjective for $\sigma \leq \tau$ large; in particular, this does not depend on the choice of the ambient category $\barBorn$ or $\barCL$. For $Q_*(f)$, the bonding maps are surjective, hence stability in $\barBorn$ (resp.~$\barCL$) is equivalent to the bonding maps being coarse (resp.~quasi-isometric) embeddings for $\sigma \leq \tau$ large.

   \begin{lemma}\label{lem:Q-stab}
    Assume that $Q_*(f)$ stabilises in $\barBorn$. Then for $\sigma$ large, $f_\sigma \from Q_\sigma(f) \to Y$ admits a lower control.
   \end{lemma}

   \proof
   Assume $Q_*(f)$ stabilises at threshold $\sigma_0$ in $\barBorn$. Then for each $\tau \geq \sigma\geq \sigma_0$, the map $Q_\sigma(f) \twoheadrightarrow Q_\tau(f)$ admits a lower control. Therefore, for each $\tau \geq \sigma$ there exists $r(\tau) \geq 0$ such that for all $x,x' \in X$,  whenever $d_{Q_\tau(f)}(x,x') \leq 1$ holds then $d_{Q_\sigma(f)}(x,x') \leq r(\tau)$. Consequently,
   \[d_Y(fx,fx') \leq \tau \iff (x,x') \in K_\tau(f) \implies d_{Q_\tau(f)}(x,x') \leq 1 \implies d_{Q_\sigma(f)}(x,x') \leq r(\tau) \]
   for all $\tau \geq \sigma$ and $x,x' \in X$, hence $f_\sigma$ admits a lower control.
   \endproof

   \begin{lemma}\label{lem:regular-coarse}
   Let $\bar f \from X \to Y$ be a morphism in $\barcat$. Let $f$ be any representative of $\bar f$. Then the following are equivalent:
   \begin{enumerate}
    \item $\bar f$ is a regular epimorphism in $\barcat$,
    \item $\bar f_\sigma \from Q_\sigma (f) \to Y$ is an isomorphism in $\barcat$ for $\sigma$ large, and
    \item $K_\sigma(f) \rightrightarrows X \xrightarrow{\bar f} Y$ is a coequaliser diagram in $\barcat$ for $\sigma$ large.
   \end{enumerate}
   Moreover, if any of the above hold, then $Q_*(f)$ stabilises in $\barcat$.
  \end{lemma}

  \proof
  The implications (2) $\implies$ (3) and (3) $\implies$ (1) are immediate.

  (1) $\implies$ (2).
    Assume $\bar f$ is the coequaliser of some parallel pair $\bar g, \bar h \from W \rightrightarrows X$ in $\barcat$. Choose representatives $g\in \bar g, h\in \bar h$. Since $\bar f \bar g = \bar f \bar h$, there exists some $\kappa \geq 0$ such that $fg \approx_\kappa fh$. Let $\sigma \geq \kappa$. Then $(gw, hw) \in K_\kappa(f) \subseteq K_\sigma(f)$ for all $w \in W$ and so $q_\sigma g \approx_1 q_\sigma h$. Consequently, there exists a unique morphism $\bar \phi_\sigma \from Y \to Q_\sigma(f)$ such that $\bar q_\sigma = \bar \phi_\sigma \bar f$.  We claim that $\bar f_\sigma = \bar \phi_\sigma^{-1}$ in $\barcat$.
    \begin{center}
    \begin{tikzcd}[row sep = large, column sep = large]
    W \arrow[r, "\bar g", shift left] \arrow[r, "\bar h"', shift right] & X \arrow[r, "\bar f"] \arrow[dr,"\bar q_\sigma"'] & Y \arrow[d,"\bar \phi_\sigma"', dashed, bend right = 20]\\
    & & Q_{\sigma}(f) \arrow[u,"\bar f_\sigma"', bend right = 20]
    \end{tikzcd}
    \end{center}

   Observe that $\bar f_\sigma \bar \phi_\sigma \bar f = \bar f_\sigma \bar q_\sigma = \bar f$ satisfies $\bar f_\sigma \bar \phi_\sigma \bar f \bar g = \bar f_\sigma \bar \phi_\sigma \bar f \bar h$. Appealing to the universal property of $\bar f$, we deduce that $\bar f_\sigma \bar \phi_\sigma = \bar 1_Y$ and so $\bar \phi_\sigma$ admits a left-inverse $\bar f_\sigma$. As $\bar q_\sigma$ is epic, so is $\bar \phi_\sigma$. Since left-invertible epimorphisms are isomorphisms, it follows that $\bar f_\sigma = \bar \phi_\sigma^{-1}$.

  The moreover statement then follows since $\bar f_\tau \bar f_\sigma^{-1} \from Q_\sigma(f) \to Q_\tau(f)$ is an isomorphism for $\sigma \leq \tau$ large.
  \endproof

  With the additional assumption that $f$ is coarsely surjective (i.e. $\bar f$ is an epimorphism), stability of $Q_*(f)$ characterises the regular epimorphism in $\barBorn$.

  \begin{proposition}\label{prop:reg-coarse}
   An epimorphism $\bar f \from X \to Y$ in $\barBorn$ is a regular epimorphism if and only $Q_*(f)$ stabilises in $\barBorn$ for any representative $f$.
  \end{proposition}

  \proof
  The forwards implication follows from Lemma \ref{lem:regular-coarse}. For the converse, the assumption that $\bar f$ is epic means that $f$ is coarsely surjective. Therefore, by Lemma \ref{lem:Q-stab}, $f_\sigma$ is a coarse equivalence for $\sigma$ large. Consequently, $\bar f_\sigma$ is an isomorphism in $\barBorn$ for $\sigma$ large by Proposition \ref{prop:morphisms}. The result follows using Lemma \ref{lem:regular-coarse}.
  \endproof

   \begin{corollary}\label{cor:reg-epi-enlarge}
   Any regular epi in $\barCL$ is also a regular epi in $\barBorn$. \qed
  \end{corollary}

  \subsection{Effective epis}

  We characterise the existence of kernel pairs in $\barcat$ using the coarse kernel filtration.
  \begin{lemma}\label{lem:K-stab}
   A morphism $\bar f \from X \to Y$ in $\barcat$ admits a kernel pair if and only if $K_*(f)$ stabilisies for any representative $f$. In that case, $\bar\pi_1\bar\iota_\sigma, \bar\pi_2\bar\iota_\sigma \from K_{\sigma}(f) \rightrightarrows X$ is the kernel pair of $\bar f$ in $\barcat$ for $\sigma$ large.
  \end{lemma}

  \proof
  By Proposition \ref{prop:equaliser}, $\Eq_*(f\pi_1, f\pi_2) = K_*(f)$ stabilises if and only if the equaliser of $\bar f\bar\pi_1, \bar f\bar\pi_2$ exists in $\barcat$; in that case $K_\sigma(f) = \Eq_\sigma(f\pi_1, f\pi_2) \hookrightarrow X \times X$ realises the equaliser for $\sigma$ large. The result then follows from Propositon \ref{prop:kernel-eq}.
  \endproof

  \begin{lemma}\label{lem:KQ-stab}
   If $K_*(f)$ stabilises at threshold $\sigma$ in $\barcat$ then so does $Q_*(f)$.
  \end{lemma}

  \begin{proof}
    Assume $K_*(f)$ stabilises at threshold $\sigma \geq 0$ and let $\tau \geq \sigma$. Then there exists a natural isomorphism between the pairs $\bar \pi_1 \bar\iota_\sigma, \bar \pi_2 \bar\iota_\sigma \from K_{\sigma}(f) \rightrightarrows X$ and $\bar \pi_1 \bar\iota_\tau, \bar \pi_2 \bar\iota_\tau \from K_{\tau}(f) \rightrightarrows X$ in $\barcat$ whose components are represented by the inclusion $K_\sigma(f) \hookrightarrow K_\tau(f)$ and identity map $1_X \from X \to X$.
    \begin{center}
        \begin{tikzcd}[row sep = normal, column sep = large]
        K_{\sigma}(f) \arrow[r, "\bar \pi_1 \bar\iota_\sigma", shift left] \arrow[r, "\bar \pi_2 \bar \iota_\sigma"', shift right] \arrow[d]  & X \arrow[d, "\bar 1_X"] \arrow[r,"\bar q_\sigma"] & Q_{\sigma}(f) \arrow[d,"\bar \phi", dashed]\\
        K_{\tau}(f) \arrow[r, "\bar \pi_1 \bar\iota_\tau", shift left] \arrow[r, "\bar \pi_2 \bar \iota_\tau"', shift right] & X  \arrow[r,"\bar q_\tau"'] & Q_{\tau}(f)
        \end{tikzcd}
    \end{center}
    Now, $\bar q_\sigma \from X \to Q_\sigma(f)$ and $\bar q_\tau \from X \to Q_\tau(f)$ are coequalisers of these respective parallel pairs. Therefore, this natural isomorphism induces a canonical morphism $\bar\phi \from Q_\sigma(f) \to Q_\tau(f)$ in $\barcat$ which, moreover, is an isomorphism. By construction, the bonding map $Q_\sigma(f) \to Q_\tau(f)$ (coinciding with the underlying identity of $X$) from the filtration $Q_*(f)$  makes the right-hand square in the diagram commute, and is hence a representative of $\bar\phi$. Thus $Q_*(f)$ stabilises.
   \end{proof}

  \begin{proposition}\label{prop:eff-coarse}
  An epimorphism $\bar f \from X \to Y$ is an effective epimorphism in $\barBorn$ if and only if $K_*(f)$ stabilises for any representative $f$.
    \end{proposition}

  \proof
  The forwards implication follows from Lemma \ref{lem:K-stab}. Conversely, assume that $\bar f$ is an epimorphism and that $K_*(f)$ stabilises. Then $Q_*(f)$ stabilises in $\barBorn$ by Lemma \ref{lem:KQ-stab}, hence $\bar f$ is a regular epimorphism in $\barBorn$ by Proposition \ref{prop:reg-coarse}. By Lemma \ref{lem:K-stab}, $\bar f$ admits a kernel pair and so it must be an effective epimorphism in $\barBorn$ by Lemma \ref{lem:reg-kernel}.
  \endproof

  \begin{example}[Coarse quotients]\label{ex:zhang}
  A coarsely surjective controlled map $f \from X \to Y$ is a \emph{coarse quotient} (in the sense of Zhang \cite{Zha15}) with constant $R \geq 0$ if for all $\epsilon \geq 0$, there exists $\delta \geq 0$ such that $N_\epsilon(fx) \subseteq N_R f(N_\delta(x))$ for all $x \in X$. For such a map $f$, suppose $(x,x') \in K_\epsilon(f)$ for some $\epsilon \geq 0$. Then $fx' \in N_\epsilon(fx)$, and so there exists $w \in N_\delta(x)$ such that $d_Y(fw, fx') \leq R$. Since $(w,x') \in K_R(f)$, it follows that $(x,x') \in N_\delta K_R(f)$. Thus, $K_*(f)$ stabilises, hence $\bar f$ is an effective epi in $\barBorn$.
  \end{example}

  \section{Epimorphisms in the coarsely Lipschitz category}

  \subsection{Regular, strong, and extremal epis}

  We now work in the category $\barCL$. Our goal is to prove that the extremal epis are regular epis, and to characterise these in terms of a relative maximality property on the target metric. Define an ordering on the set of metrics on a set $Y$ as follows: declare $d \prec d'$ if and only if the underlying identity $(Y,d') \to (Y,d)$ is coarsely Lipschitz. This is analogous to the ordering of compatible left-invariant metrics on a topological group in the sense of Rosendal \cite{Ros22}.

  Our strategy is to adapt the weighted Rips graph construction from \cite{Tang-rips} by adding augmented edges. Let $f \from X \to Y$ be a coarsely Lipschitz map and assume that $N_r f(X) = Y$ for some $r \geq 0$.     Let $\Theta \from [0,\infty) \to [1,\infty)$ be an increasing function,    called a \emph{weight function}. Define a weighted graph $\Rips^\Theta_\infty(Y;f)$ with vertex set $Y$, and with edges coming in two types:
  \begin{itemize}
   \item \emph{internal edges} -- for each distinct pair $y,y' \in Y$ such that $d_Y(y,y') < \infty$, add a edge of weight $\Theta({d_Y(y,y')})$,
   \item \emph{augmented edges} -- for each distinct pair $y,y' \in f(Y)$ such that $d_Y(y,y') < \infty$, add an edge of weight $d_X(f^{-1}(y), f^{-1}(y')) + 1$.
  \end{itemize}

  \textbf{Augmented weighted Rips filtration.}
  For $\sigma \geq 0$, define $\Rips_\sigma^\Theta(Y;f)$ to be the subgraph of $\Rips_\infty^\Theta(Y;f)$ with vertex set $Y$, and whose edges comprise all augmented edges together with those internal edges between distinct pairs $y,y' \in Y$ satisfying $d_Y(y,y') \leq \sigma$. These assemble via the inclusions $\Rips_\sigma^\Theta (Y;f) \hookrightarrow \Rips_\tau^\Theta(Y;f)$ for $\sigma \leq\tau$ to form a filtration $\Rips_*^\Theta(Y;f)$.

  We shall equip $\Rips_\infty^\Theta(Y;f)$ (resp.~$\Rips_\sigma^\Theta (Y;f)$) with the path metric. Let $Y_\infty^\Theta = (Y, \partial_\infty^\Theta)$ (resp. $Y_\sigma^\Theta = (Y, \partial_\sigma^\Theta)$) denote its vertex set endowed with the induced metric. Note that the inclusions $\Rips_\sigma^\Theta (Y;f) \hookrightarrow \Rips_\tau^\Theta(Y;f) \hookrightarrow \Rips_\infty^\Theta(Y;f)$ are 1--Lipschitz, thus $\partial_\infty \prec \partial^{\tau} \prec \partial^\sigma$.               Let $\Gamma_\sigma^\Theta$ be the complete subgraph of $\Rips_\sigma^\Theta(Y; f)$ spanned by $f(X) \subseteq Y$, equipped with its path metric, and $U_\sigma^\Theta$ be its vertex set with the induced metric.

  Consider the factorisation $X \xrightarrow{q_\sigma} Q_\sigma(f) \xrightarrow{f_\sigma} Y$ of $f$ for $\sigma \geq 0$. The map $f_\sigma$ itself factorises as
  \[Q_\sigma(f) \twoheadrightarrow U_\sigma^\Theta \hookrightarrow Y_\sigma^\Theta \to Y^\Theta_\infty \to Y,\]
  where the first arrow coincides with $f$ onto its image, and the others with the underlying inclusion or identity maps. We will give sufficient conditions for each arrow to represent a mono in $\barCL$ for $\sigma$ large in terms of the weight function $\Theta$. (For concreteness, consider $\Theta(t) = 2^t$ or $\Theta(t) = 1$.) Extremality of $\bar f$ will then imply that $\bar f_\sigma$ is an isomorphism, from which regularity would follow using Lemma \ref{lem:regular-coarse}.

  \begin{lemma}\label{lem:ext-qi}
   For all $\sigma \geq 0$, the map $Q_\sigma(f) \twoheadrightarrow U^\Theta_\sigma$ induced by $f$ is a quasi-isometry.   \end{lemma}

  \proof
  Let us regard $Q_\sigma(f)$ as the vertex set of a weighted graph $\Gamma$ with \emph{internal} and \emph{glued} edges (as in the coarse gluing construction). Let $x,x' \in X$ and $y = fx, y'=fx'$. To obtain an upper control, suppose $d_{Q_\sigma(f)}(x,x') < \infty$. Let $P$ be an edge-path in $\Gamma$ from $x$ to $x'$ of length $L \geq 0$. We may assume that all vertices $x = x_0, \ldots, x_n = x$ appearing along $P$ are distinct and that $P$ has no consecutive pair of internal edges. Thus $n \leq L$ since glued edges have unit length. Let $y_i = fx_i$. If the edge of $P$ between $x_i,x_{i+1}$ is
  \begin{itemize}
   \item a glued edge then, by definition of $Q_\sigma(f)$, we have that $d_Y(y_i, y_{i+1}) \leq \sigma$, hence $y_i, y_{i+1}$ are connected by an internal edge of weight at most $\Theta(\sigma)$ in $\Gamma^\Theta_\sigma$,
   \item an internal edge then $y_i, y_{i+1}$ are connected by an augmented edge of weight at most $d_X(f^{-1}(y_i), f^{-1}(y_{i+1})) + 1 \leq d_X(x_i,x_{i+1}) + 1$ in $\Gamma^\Theta_\sigma$.
  \end{itemize}
   Therefore, there exists an edge-path from $y$ to $y'$ in $\Gamma^\Theta_\sigma$ of length at most
  \[\sum_{i=0}^{n-1} \left(d_X(x_i,x_{i+1}) + \Theta(\sigma)\right) = L +  \Theta(\sigma) n \leq (\Theta(\sigma) + 1)L. \]   By considering all paths from $x$ to $x'$, we deduce that $Q_\sigma(f) \to U^\Theta_\sigma$ is $(\Theta(\sigma) + 1)$--Lipschitz.

  For the lower control, suppose that   $P'$ is an edge-path from $y$ to $y'$ in $\Gamma^\Theta_\sigma$ of length $L \geq 0$. We may assume all vertices $y = y_0,\ldots, y_n = y$ along $P'$ are distinct. Since all edges of $\Gamma^\Theta_\sigma$ have weight at least 1, we deduce $n \leq L$.     Note that $f^{-1}(y_i)$ is non-empty for all $i$. If the edge of $P'$ between $y_i, y_{i+1}$ forms
  \begin{itemize}
   \item an augmented edge then there exist $x'_i \in f^{-1}y_i$ and $x_{i+1} \in f^{-1}y_{i+1}$ such that
   \[d_X(x'_i,x_{i+1}) \leq d_X(f^{-1}(y_i), f^{-1}(y_{i+1})) + 1,\]
   \item an internal edge then $d_Y(y_i, y_{i+1}) \leq \sigma$ by construction; in particular, every pair of distinct points in $f^{-1}y_i \cup f^{-1}y_{i+1}$ are connected by a glued edge in $Q_\sigma(f)$.
  \end{itemize}
   In either case, if this edge between $y_i, y_{i+1}$ has weight $w$ then there exists vertices $x'_i \in f^{-1}y_i$, $x_{i+1} \in f^{-1}y_{i+1}$ connected by an edge of weight at most $w$ in $\Gamma$. Therefore, there is a path in $\Gamma$ passing through vertices $x, x'_0, x_1, \ldots, x'_{n-1}, x_n, x'$ of length at most $L + n + 1 \leq 2L + 1$. Consequently, $Q_\sigma(f) \to Y^\Theta_\sigma$ has lower control $t \mapsto \frac{t - 1}{2}$.
  \endproof

  For the next arrow, we ask that $\Theta$ grows at most exponentially.

  \begin{lemma}\label{lem:image-qi}
        Assume there exists $C \geq 1$ such that $\Theta(t + 2r) \leq C\Theta(t)$ for all $t \geq 0$. Then for all $\sigma \geq r$, the inclusion $\iota \from U_\sigma^\Theta \hookrightarrow Y_\sigma^\Theta$ is a quasi-isometry.
   \end{lemma}

   \proof
   For each $y \in Y$, choose some $\phi(y) \in f(X)$ such that $d_Y(y, \phi y) \leq r$; we shall also assume that $\phi u = u$ for all $u \in f(X)$.      We claim that $\phi \from Y_\sigma^\Theta \to U_\sigma^\Theta$ is $C$--Lipschitz. Since the involved metrics are defined using path metrics, it suffices to check this on edges. Observe that any augmented edge in $\Rips_\sigma^\Theta(Y; f)$ also belongs to the subgraph $\Gamma_\sigma^\Theta$. Now suppose $y,y' \in Y$ are connected by an internal edge of weight $w$. Then $d_Y(\phi y, \phi y') \leq d_Y(y,y') + 2r < \infty$, hence $\phi y, \phi y'$ are connected by an internal edge of weight
   \[\Theta(d_Y(\phi y, \phi y')) \leq \Theta(d_Y(y,y') + 2r) \leq C \Theta(d_Y(y,y')) = Cw\]
   in $\Gamma_\sigma^\Theta$, yielding the claim. Now, $\iota$ is 1--Lipschitz and $\phi \iota$ is the identity on $U_\sigma^\Theta$, hence $\iota$ is a bi-Lipschitz embedding. Since $N_r f(X) = Y$ and $\sigma \geq r$, each $y \in Y$ is connected to some $u \in f(X)$ be an internal edge of weight at most $\Theta(r)$ in $\Gamma_\sigma^\Theta$, and so $\iota$ is coarsely surjective.
   \endproof

  For the remaining arrows, we require lower bounds on the growth of $\Theta$.

  \begin{lemma}\label{lem:lower}
   Assume $t \leq \Theta(t)$ for all $t \geq 0$. Then the map $Y^\Theta_\infty \to Y$ induced by the identity is coarsely Lipschitz and admits a lower control.
  \end{lemma}

  \proof
  Let $\rho$ be an (affine) upper control for $f$. Suppose that $y,y' \in Y$ are connected by an edge of weight $w$ in $\Rips^\Theta_\infty(Y;f)$. If this is an internal edge then $d_Y(y,y') \leq \Theta(d_Y(y,y')) = w$;   otherwise, this is an augmented edge, and so there exist $x \in f^{-1}(y)$ and $x' \in f^{-1}(y')$ such that $d_X(x,x') \leq d_X(f^{-1}y, f^{-1}y') + 1 = w$, hence
  \[d_Y(y,y') \leq \rho(d_X(x,x')) \leq \rho(w).\] Therefore, $t \mapsto \rho(t) + t$ serves as an upper control for $Y'_\infty \to Y$.

  If $d_Y(y,y') < \infty$ then $\partial_\infty(y,y') \leq \Theta(d_Y(y,y'))$ since there exists an internal edge between $y,y'$. Moreover, $\partial_\infty(y,y') = \infty$ only if $d_Y(y,y') = \infty$ by construction. Since $\Theta$ is proper, we obtain a lower control.  
  \endproof

  \begin{lemma}\label{lem:ext-stab}
   Assume that for every strictly increasing affine function $\theta \from [0,\infty) \to \R$, there exists $\sigma \geq 0$ such that $t < \theta(\Theta(t) - 1)$ for all $t \geq \sigma$. If $\bar f \from X \to Y$ is an extremal epi in $\barCL$ then for $\sigma \geq 0$ large, the underlying identity induces an isometry $Y^\Theta_\sigma \to Y^\Theta_\infty$.
  \end{lemma}

  \proof
  By Lemma \ref{lem:lower}, the factorisation $X \to Y
  ^\Theta_\infty \to Y$ of $f$ yields an epi-mono factorisation of $\bar f$ in $\barCL$. Extremality of $\bar f$ implies that $Y^\Theta_\infty \to Y$ represents an isomorphism in $\barCL$ and is hence a quasi-isometry. Therefore, this map admits a (strictly increasing) affine lower control $\theta$. By assumption, there exists $\sigma \geq 0$ such that $t < \theta(\Theta(t) - 1)$ for all $t \geq \sigma$.

  Let $y,y' \in Y$ satisfy $\partial^\Theta_\infty(y,y') < \infty$. Consider a weighted edge-path $P$ in $\Rips_\infty^\Theta(Y;f)$ passing through vertices $y = y_0, \ldots, y_n = y'$ of length $L \in [\partial^\Theta_\infty(y,y'), \partial^\Theta_\infty(y,y') + \frac{1}{3})$. Note that each $d_Y(y_i,y_{i+1})$ is finite. If there exists some $0 \leq i \leq n$ such that $d_Y(y_i,y_{i+1}) \geq \sigma$ then
  \[\theta\partial_\infty^\Theta(y_i,y_{i+1}) \leq d_Y(y_i,y_{i+1}) < \theta\left(\Theta {d_Y(y_i,y_{i+1})} - 1\right),\]
  hence $\partial_\infty^\Theta(y_i,y_{i+1}) < \Theta {d_Y(y_i,y_{i+1})} - 1$.  Therefore, $P$ does not run over the internal edge between $y_i$ and $y_{i+1}$, for otherwise we could replace that edge with a detour in $\Rips_\infty^\Theta(Y;f)$ which reduces the length of $P$ by at least $\frac{1}{2}$, contradicting the choice of $P$. In other words, each edge appearing along $P$ must be an edge in the subgraph $\Rips_\sigma^\Theta(Y;f)$. Thus, for all $y,y'\in Y$, the distance $\partial_\infty(y,y')$ is attained by the infimal length of all weighted edge-paths from $y$ to $y'$ in $\Rips_\sigma^\Theta(Y;f)$. Consequently, $Y^\Theta_\sigma \to Y^\Theta_\infty$ is an isometry.
  \endproof
  
  We now characterise the extremal (and regular) epis in $\barCL$ in terms of relatively maximal metrics on the codomain. Say that $d_Y$ is \emph{maximal relative to} $f$ if condition (4) below holds.

  \begin{proposition}\label{prop:ext-reg}
  Let $\bar f \from (X,d_X) \to (Y,d_Y)$ be an epi in $\barCL$ and let $f$ be a representative. Then the following are equivalent.
  \begin{enumerate}
   \item $\bar f$ is an extremal epi in $\barCL$,
   \item $\bar f$ is a regular epi in $\barCL$,
   \item the metric $d_Y$ is $\prec$--greatest among all coarsely equivalent metrics $d'$ on $Y$ such that the map $f' \from X \to (Y,d')$ coinciding with $f$ is coarsely Lipschitz, and
   \item the metric $d_Y$ is $\prec$--maximal among all coarsely equivalent metrics $d'$ on $Y$ such that the map $f' \from X \to (Y,d')$ coinciding with $f$ is coarsely Lipschitz.
  \end{enumerate}
  \end{proposition}

  \proof
  (1) $\implies$ (2).
  Consider the factorisation $X \xrightarrow{q_\sigma} Q_\sigma(f) \xrightarrow{f_\sigma} Y$ of $f$ for $\sigma \geq 0$. The weight function $\Theta(t) = 2^t$ satisfies the hypotheses of Lemmas \ref{lem:ext-qi}, \ref{lem:image-qi}, \ref{lem:lower}, and \ref{lem:ext-stab}, hence $\bar f_\sigma$ is monic in $\barCL$ for $\sigma$ large. Since $\bar f$ is an extremal epi, $\bar f_\sigma$ is an isomorphism in $\barCL$, hence $\bar f$ is a regular epi by Lemma \ref{lem:regular-coarse}.

  (2) $\implies$ (3). Let $d'$ be a metric on $Y$ such that the map $m \from (Y,d') \to (Y,d_Y)$ given by the underlying identity is a coarse equivalence and $f' \from X \to (Y,d')$ is coarsely Lipschitz. By Lemma \ref{lem:regular-coarse}, $K_\sigma(f) \rightrightarrows X \xrightarrow{\bar f} (Y,d)$ is a coequaliser diagram in $\barCL$ for $\sigma$ large. Note that $\bar f' = \bar m \bar f$ coequalises the parallel pair $K_\sigma(f) \rightrightarrows X$. Therefore, by the universal property, there exists a unique morphism $\bar g \from (Y,d_Y) \to (Y,d')$ in $\barCL$ such that $\bar f' = \bar g \bar f$. Since $\bar f$ is an epi in $\barBorn$, we deduce that $\bar g = \bar m$. Therefore, $m$ is coarsely Lipschitz, hence $d' \prec d_Y$.

  (3) $\implies$ (4). Trivial.

  (4) $\implies$ (1). Assume $\bar f$ factorises as $X \xrightarrow{\bar e} Z \xrightarrow{\bar m} Y$ of $\bar f$ in $\barCL$, with $\bar m$ a mono. Since $\bar f$ is an epi, so is $\bar m$. Thus, $m \from Z \to Y$ is a coarse equivalence. By Lemma \ref{lem:coarse-qi}, $m$ factorises as $Z \xrightarrow{g} (Y,d') \xrightarrow{m'} (Y,d)$ for some metric $d'$ with $g$ a quasi-isometry and $m'$ a coarse equivalence coinciding with the underlying identity. Note that $\bar m' = \bar g^{-1} \bar m$ is a morphism in $\barCL$. Now, $f' := ge$ is coarsely Lipschitz and coincides with $f$ on underlying sets. By the maximality assumption, $(m')^{-1}$ is coarsely Lipschitz, hence $\bar m'$ is an isomorphism in $\barCL$. Therefore, $\bar m = \bar g \bar m'$ is an isomorphism in $\barCL$ and so $\bar f$ is an extremal epi.
  \endproof

   \begin{example}
   By \cite[Theorem 1.8]{Tang-rips}, a metric space $(Y,d)$ is quasigeodesic (or, equivalently, quasi-isometric to a graph with the combinatorial metric) if and only if $d$ is $\prec$--maximal among all coarsely equivalent metrics $d'$ on $Y$. Therefore, any epimorphism in $\barCL$ to a quasigeodesic space is extremal. For a partial converse, take $X$ to be the space with underlying set $Y$ equipped with the metric $d_X(y,y') = \infty$ for all distinct $y,y' \in Y$. Then by Proposition \ref{prop:ext-reg}, if the underlying identity $f \from X \to Y$ is extremal in $\barCL$ then $Y$ is quasigeodesic.
   \end{example}

  Next, we give a criterion for promoting an epi in $\barCL$ to an extremal epi via a coarsely equivalent remetrisation of the codomain metric.

  \begin{proposition}\label{prop:promote}
   Let $f \from (X,d_X) \to (Y,d_Y)$ be a coarsely Lipschitz map and assume $N_r f(X) = Y$ for some $r \geq 0$. Then the following are equivalent.
   \begin{enumerate}
    \item $\bar f$ is a regular epi in $\barBorn$,
    \item there exists a $\prec$--greatest element among all metrics $d'$ on $Y$ coarsely equivalent to $d_Y$ such that the map $X \to (Y,d')$ coinciding with $f$ is coarsely Lipschitz, and
    \item there exists a $\prec$--maximal element among all metrics $d'$ on $Y$ coarsely equivalent to $d_Y$ such that the map $X \to (Y,d')$ coinciding with $f$ is coarsely Lipschitz, and
    \item $\bar f$ factors as $\bar f = \bar m \bar e$ with $\bar e$ an extremal epi and $\bar m$ a mono in $\barCL$.
   \end{enumerate}
   If any of the above hold, then a relatively maximal metric is realised by $\partial_\sigma^\Theta$ for $\sigma$ large, where $\Theta$ is any weight function satisfying the hypothesis of Lemma \ref{lem:image-qi}.
   \end{proposition}

  \proof
  (1) $\implies$ (2).
  Let $\Theta$ be any weight function satisfying the hypothesis of Lemma \ref{lem:image-qi}. For $\sigma \geq 0$, consider the factorisation
  \[X \xrightarrow{q_\sigma} Q_\sigma(f) \xrightarrow{g_\sigma} (Y, \partial^\Theta_\sigma) \xrightarrow{h_\sigma} (Y,d_Y)\]
  of $f$, where $g_\sigma$ and $h_\sigma$ respectively coincide with  $f$ and the identity as functions. By construction, $\bar q_\sigma \from X \to Q_\sigma(f)$ is a regular epi in $\barCL$ for all $\sigma \geq 0$. Appealing to Lemmas \ref{lem:ext-qi}, \ref{lem:image-qi}, and coarse surjectivity of $f$, we deduce that $g_\sigma$ is a quasi-isometry, hence $\bar g_\sigma$ is an isomorphism in $\barCL$. Therefore, $\bar g_\sigma \bar q_\sigma$ is a regular epi in $\barCL$. By assumption, $\bar f$ is a regular epi in $\barBorn$ and so, by Proposition \ref{prop:reg-coarse}, we may choose $\sigma$ large so that $f_\sigma = h_\sigma g_\sigma \from Q_\sigma(f) \to Y$ is a coarse equivalence. In particular, $\partial_\sigma^\Theta$ and $d_Y$ are coarsely equivalent. Suppose that $d'$ is a metric on $Y$ such that the map $k \from (Y,d_Y) \to (Y,d')$ given by the underlying identity is coarse equivalence, and that $f':= kf \from X \to (Y,d')$ is coarsely Lipschitz.  Since $\bar g_\sigma \bar q_\sigma$ is an extremal epi in $\barCL$, it follows from Proposition \ref{prop:ext-reg} that $d' \prec \partial^\Theta_\sigma$.

  (2) $\iff$ (3). This follows from Proposition \ref{prop:ext-reg}.

  (2) $\implies$ (4). Let $d'$ be a metric on $Y$ satisfying the given $\prec$--greatest condition. Then, by Proposition \ref{prop:ext-reg}, $X \to (Y,d') \to (Y,d_Y)$ yields an extremal epi--mono factorisation of $\bar f$ in $\barCL$.

  (4) $\implies$ (1).
  By Corollary \ref{cor:reg-epi-enlarge}, $\bar e$ is a regular epi in $\barBorn$. Since $\bar f$ is an epi in $\barBorn$, so is $\bar m$. Therefore, $\bar m$ is an isomorphism in $\barBorn$. Consequently, $\bar f$ is a regular epi in $\barBorn$.
  \endproof

   If such a relatively maximal metric exists then it is unique up to quasi-isometry; this implies that the filtration $Y_*^\Theta$ stabilises in $\barCL$. Appealing to Proposition \ref{prop:reg-coarse} and Lemma \ref{lem:ext-qi}, we deduce:

  \begin{corollary}
   Let $\bar f$ be an epi in $\barCL$ and $f$ be a representative. Then $Q_*(f)$ stabilises in $\barBorn$ if and and only if it stabilises in $\barCL$. \qed
  \end{corollary}

  \begin{remark}
   Proposition \ref{prop:promote} does not require properness of $\Theta$. Indeed, by choosing $\Theta(t) = 1$, we may construct a $\prec$--maximal metric $\partial_\sigma^\Theta$ relative to $f$ by taking the standard Rips graph $\Rips_\sigma Y$ at sufficiently large scale $\sigma$, then adding the augmented edges.
  \end{remark}

  \subsection{Effective epis}

  \begin{proposition}\label{prop:eff-clip}
   A morphism $\bar f \from X \to Y$ in $\barCL$ is an effective epimorphism in $\barCL$ if and only $\bar f$ is a regular epimorphism in $\barCL$ and $K_*(f)$ stabilises for any representative $f$.
  \end{proposition}

  \proof
   By Lemma \ref{lem:K-stab}, $\bar f$ admits a kernel pair in $\barCL$ if and only if $K_*(f)$ stabilises for any representative $f$. The result follows using Lemma \ref{lem:reg-kernel}.
  \endproof

  \section{Non-implications and examples}

  \subsection{Strong does not imply regular in \texorpdfstring{$\barBorn$}{$\mathsf{Coarse}$}}

  \begin{example}[Comb graph]\label{ex:comb}
  Construct a graph in the integer lattice $\Z^2$ as follows. Begin with the infinite ray $[0,\infty) \times \{0\}$. Set current position at $(x,0) = (0,0)$. Then do the following:
  \begin{itemize}
   \item For each integer $n = 1, 2, \ldots$, do:
   \begin{itemize}
   \item For $j = 1, \ldots, n$, do: attach a vertical line segment from $(x,0)$ to $(x,n)$, then move current position to $(x+j, 0)$.
   \end{itemize}
  \end{itemize}
   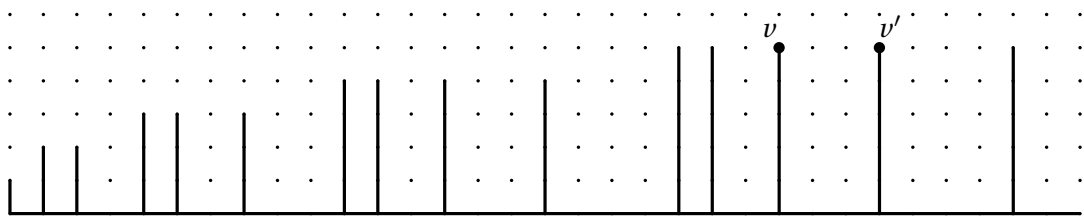
\begin{figure}[h]\label{fig:comb}
  \begin{center}
  \begin{tikzpicture}[scale=0.44,every path/.style= very thick]
        \draw (0,0) -- (32,0);
    \coordinate (current) at (0,0);
    \coordinate (xpos) at (0,0);

        \foreach \n in {1,...,5} {

                \foreach \j in {1,...,\n} {
                        \draw ($(xpos) + (0,0)$) -- ($(xpos) + (0,\n)$);

                        \coordinate (xpos) at ($(xpos) + (\j,0)$);
        }
    }

        \foreach \x in {0,1,...,32} {
        \foreach \y in {0,1,...,6} {
            \fill (\x,\y) circle (0.15em);
        }
    }

    \node at (22.7,4.7) [label=above:$v$] {};
    \node at (26.3,4.7) [label=above:$v'$] {};
    \fill (23,5) circle (0.5em);
    \fill (26,5) circle (0.5em);
\end{tikzpicture}
  \end{center}
  \caption{A portion of the comb graph with vertices labelled $v,v'$ as in the case $\sigma = 2$ and $n = 5$.}
  \end{figure}

  Let $X$ be the vertex set of this graph with the standard path metric. Let $Y$ be the underlying set of $X$ equipped with the induced $\ell^1$--metric from $\Z^2$. Let $f \from X \to Y$ be the underlying identity. By Proposition \ref{prop:strong-coarse}, $\bar f$ is a strong epi in $\barBorn$. By construction, for all integers $\sigma, n \geq 1$, there exists an integer $x \geq 0$ such that $v=(x,n)$ and $v'=(x + \sigma + 1, n)$ are vertices of $X$ appearing on consecutive vertical segments; see Figure \ref{fig:comb}. Since $d_Y(v,v') = \sigma + 1$, we deduce that $d_{Q_{\sigma+1}(f)}(v,v') = 1$. Now, any path from $v$ to $v'$ in $Q_\sigma(f)$ must pass through some vertex $u$ on the horizontal segment between $(x,0)$ and $(x + \sigma + 1, 0)$. It follows that
  \[d_{Q_\sigma(f)}(v,v') \geq d_{Q_\sigma(f)}(v,u) + d_{Q_\sigma(f)}(u,v') \geq \frac{2n}{\sigma}.\]
  Since $n \geq 1$ is arbitrary, we deduce that $Q_\sigma(f) \to Q_{\sigma+1}(f)$ is not a coarse equivalence. Therefore, $Q_*(f)$ does not stabilise in $\barBorn$, and so $\bar f$ is not a regular epi by Proposition \ref{prop:reg-coarse}.
  \end{example}

  \begin{remark}
  This also yields an example of an epimorphism in $\barCL$ which does not admit a relatively maximal metric on the codomain, as in Proposition \ref{prop:promote}.
  \end{remark}

  \subsection{Epi does not imply extremal in \texorpdfstring{$\barCL$}{$\mathsf{CLip}$}}

  \begin{example}[Cubes to squares]\label{ex:cubes}
   Let $X = \set{n^3 \st n \in \N}$ and $Y = \set{n^2 \st n \in \N}$ equipped with the Euclidean metrics. Define $f \from X \to Y$ by $n^3 \mapsto n^2$. Then $f$ is a coarse equivalence but not a quasi-isometry. Thus $\bar f$ is monic and epic in $\barCL$, but not an isomorphism. Since $\bar f$ factorises as $\bar f =  \bar f \bar 1_X$, it follows that $\bar f$ is not an extremal epi in $\barCL$
  \end{example}

  \begin{remark}
   Example \ref{ex:cubes} also shows that an isomorphism in $\barBorn$ which is coarsely Lipschitz is not necessarily an extremal epi in $\barCL$.
  \end{remark}

  \subsection{Split does not imply effective}

  \begin{example}[Comb retraction]\label{ex:comb-retract}
   Let $X$ be the comb graph as given in Example \ref{ex:comb}. Equip $[0,\infty)$ with the Euclidean metric and define a 1--Lipschitz map $f \from X \to [0,\infty)$ by $(x,y) \mapsto x$.    Observe that $\bar f$ is a split epi in $\barcat$ since the isometric embedding $x \mapsto (x,0)$ represents a right-inverse. Given integers $\sigma, n \geq 1$, choose vertices $v,v'$ of $X$ as in Example \ref{ex:comb}. Observe that $(v,v') \in K_{\sigma+1}(f)$. Now, any vertex $u$ (resp.~$u'$) of $X$ within distance $n$ of $v$ (resp.~$v'$) must lie on the vertical segment of $X$ containing $v$ (resp.~$v'$). In particular $|f(u) - f(u')| = \sigma + 1$ and so $(u,u')\notin K_{\sigma}(f)$. Therefore $K_{\sigma + 1}(f) \not\subseteq N_n(K_\sigma(f))$. Since $n$ is arbitrary,  $K_*(f)$ does not stabilise, hence $\bar f$ is not an effective epi in $\barcat$ by Lemma \ref{lem:K-stab}.
  \end{example}

    Another example of a split, but not effective, epi is the closest point projection $f \from \Hy^2 \to \gamma$ from the hyperbolic plane to any bi-infinite geodesic.

  \subsection{Effective does not imply split}

  Let $Z$ and $G$ be finitely generated groups. Consider a short exact sequence of groups
  \[1 \to Z \hookrightarrow E \xrightarrow{f} G \to 1.\]
  Choose finite generating sets $S_Z$ and $S_G$ for $Z$ and $G$ respectively. Let $\hat S_G \subseteq E$ be a set of lifts for $S_G$ under $f$ (i.e.,~$f$ restricts to a bijection $\hat S_G \to S_G$). Then $S_E := S_Z \cup \hat S_G$ is a finite generating set $S_E$. Equip each group with their corresponding word metrics.

  \begin{lemma}\label{lem:eff-group}
   The morphism $\bar f \from E \to G$ is an effective epimorphism in $\barcat$.
  \end{lemma}

  \proof
  First, we show that $\bar f$ is a regular epimorphism in $\barcat$.
  Observe that $(a,b) \in K_0(f)$ if and only $aZ = bZ$. Therefore, $Q_0(f)$ is the Cayley graph of $E$ with respect to the generating set $(Z \cup \hat S_G)\setminus \set{1}$.     The map $Q_0(f) \to G$ induced by $f$ collapses each coset $aZ$ to a point $f(a)$; moreover, distinct cosets $aZ$, $bZ$ map to adjacent vertices if and only if some vertex of $aZ$ is adjacent to a vertex in $bZ$. Since each coset has diameter at most 1 in $Q_0(f)$, this map is a quasi-isometry (see \cite[Lemma 4.2]{DPRT21}). Therefore, $\bar f$ is a regular epi in $\barcat$ by Lemma \ref{lem:regular-coarse}.

  Next, we show that the coarse kernel filtration $K_*(f)$ stabilises. Suppose that $(a,b) \in K_\sigma(f)$ for some $\sigma \geq 0$. Then $a,b\in E$ satisfy $d_{G}(1, f(a^{-1}b)) = d_{G}(f(a), f(b)) \leq \sigma$.     Thus $f(a^{-1}b) \in G$ can be expressed as a word in at most $\sigma$ generators. By replacing each generator in this word with its lift in $\hat S_G$, we obtain an element $h \in a^{-1}bZ$ such that $d_E(1,h) \leq \sigma$. Note that $ahZ = bZ$, hence $(ah, b) \in K_0(f)$. Since $d_E(a, ah) = d_E(1,h)\leq \sigma$, we deduce   $(a,b) \in N_\sigma(K_0(f))$. Therefore $K_*(f)$ stabilises, hence $\bar f$ is an effective epi in $\barcat$ by Propositions \ref{prop:eff-coarse} and \ref{prop:eff-clip}.
      \endproof

  Now, we show that any right inverse of $\bar f$ is realised by a Lipschitz section of $f$.

  \begin{lemma}\label{lem:true-section}
   Suppose that $\bar f \from E \to G$ admits a right inverse $\bar s \from G \to E$ in $\barcat$. Then there exists a Lipschitz representative $s \in \bar s$ such that $fs$ is the identity map on $G$.
  \end{lemma}

  \proof
  Let $s' \in \bar s$ be any representative. Then there exists $\kappa \geq 0$ such that $fs' \approx_\kappa 1_G$. For any $g \in G$, we have that $d_G(fs'g, g) \leq \kappa$. Thus, we may right-multiply $s'g$ by a word of most $\kappa$ generators from $\hat S_G$ to obtain an element $s(g) \in f^{-1}(g)$. Therefore, there exists a section $s \from G \to E$ of $f$ satisfying $s \approx_\kappa s'$. Since $s'$ is controlled, so is $s$. The result follows since any controlled map from the vertex set of a graph (with the combinatorial metric) is Lipschitz.
  \endproof

  In the case of central extensions, a result of Kleiner--Leeb shows that the existence of a Lipschitz section implies that the extension group is quasi-isometric to a product.

  \begin{proposition}[{\cite[Proposition 8.3]{KL01}}]\label{prop:qi-trivial}
   Let $1 \to Z \hookrightarrow E \xrightarrow{f} G \to 1$ be a central extension, where $Z$ and $G$ are finitely generated. Then $f$ admits a Lipschitz section $s \from G \to E$ if and only if the extension is \emph{quasi-isometrically trivial}: there exists a quasi-isometry $\phi \from E \to Z \times G$ such that $f \approx \phi \circ \pi_{G}$, where $\pi_{G} \from Z \times G \to G$ is the projection to the $G$--factor. \qed

  \end{proposition}

  \begin{example}[Integer Heisenberg group]\label{ex:heisenberg}
  Let $H = \set{\begin{pmatrix}
                 1 & x & z \\ 0 & 1 & y \\ 0 & 0 & 1
                \end{pmatrix}
                \st x,y,z \in \Z
    }$ be the integer Heisenberg group and $Z = \set{\begin{pmatrix}
                 1 & 0 & z \\ 0 & 1 & 0 \\ 0 & 0 & 1
                \end{pmatrix}
                \st z \in \Z
    } \cong \Z$. These groups fit into a central extension $1 \to Z \hookrightarrow H \xrightarrow{f} \Z^2 \to 1$. By Lemma \ref{lem:eff-group}, $\bar f$ is an effective epi in $\barcat$. Since $H$ has quartic growth rate \cite[Lemma 4]{Mil68}\cite[VII.22]{Har00}, it is not quasi-isometric to $\Z^3 \cong Z \times \Z^2$ (which has cubic growth rate). Thus $\bar f$ is not a split epi in $\barcat$ by Lemma \ref{lem:true-section} and Proposition \ref{prop:qi-trivial}.
    \end{example}

  \subsection{Proofs of the main theorems}

  We finally establish the (non-)implications that do not follow from Proposition \ref{prop:implications} and Lemma \ref{lem:split}.

  \proofof{Theorem \ref{thm:epi-coarse}}
  In $\barBorn$, the implication epi $\implies$ strong epi is Proposition \ref{prop:strong-coarse}. Example \ref{ex:comb} shows that strong epis are not necessarily regular epis. The remaining non-implications follow from Examples \ref{ex:comb-retract} and \ref{ex:heisenberg}.
  \qed

  \proofof{Theorem \ref{thm:epi-clip}}
  In $\barCL$, the implication extremal epi $\implies$ regular epi is Proposition \ref{prop:ext-reg}. Example \ref{ex:cubes} demonstrates that not every epi is an extremal epi. The remaining non-implications follow using the same examples as for the proof of Theorem \ref{thm:epi-coarse}.   \qed

  \providecommand{\bysame}{\leavevmode\hbox to3em{\hrulefill}\thinspace}
\providecommand{\MR}{\relax\ifhmode\unskip\space\fi MR }
\providecommand{\MRhref}[2]{%
  \href{http://www.ams.org/mathscinet-getitem?mr=#1}{#2}
}
\providecommand{\href}[2]{#2}


\begin{thebibliography}{EMM15}

\bibitem[AHS06]{AHS06}
J. Ad\'amek, H. Herrlich and G.~E. Strecker, Abstract and concrete categories: the joy of cats, Repr. Theory Appl. Categ. No. 17, 2006, 1--507; MR2240597

\bibitem[AV20]{AV20}
J. Aramayona and N.~G. Vlamis, Big mapping class groups: an overview, in {\it In the tradition of Thurston---geometry and topology}, 459--496, Springer, Cham, 2020; MR4264585

\bibitem[Bor94]{Bor94}
F. Borceux, {\it Handbook of categorical algebra. 1}, Encyclopedia of Mathematics and its Applications, 50, Cambridge Univ. Press, Cambridge, 1994; MR1291599

\bibitem[Bra22]{Bra22}
B.~M. Braga, Coarse quotients of metric spaces and embeddings of uniform Roe algebras, J. Noncommut. Geom. {\bf 16} (2022), no.~4, 1337--1361; MR4542387

\bibitem[CH16]{CH16}
Y. de~Cornulier and P. de~la~Harpe, {\it Metric geometry of locally compact groups}, EMS Tracts in Mathematics, \textbf{25}, Eur. Math. Soc., Z\"urich, 2016; MR3561300

\bibitem[DPRT21]{DPRT21}
V.~Disarlo, H.~Pan, A.~Randecker and R.~Tang, Large-scale geometry of the saddle connection graph, Trans. Amer. Math. Soc. {\bf 374} (2021), no.~11, 8101--8129; MR4328693

\bibitem[Har00]{Har00}
P. de~la~Harpe, {\it Topics in geometric group theory}, Chicago Lectures in Mathematics, Univ. Chicago Press, Chicago, IL, 2000; MR1786869

\bibitem[HW19]{HW19}
L. Higginbotham and T. Weighill, Coarse quotients by group actions and the maximal Roe algebra, J. Topol. Anal. {\bf 11} (2019), no.~4, 875--907; MR4040015

\bibitem[KL01]{KL01}
B. Kleiner and B. Leeb, Groups quasi-isometric to symmetric spaces, Comm. Anal. Geom. {\bf 9} (2001), no.~2, 239--260; MR1846203

\bibitem[LV23]{LV23}
A. Leitner and F. Vigolo, {\it An invitation to coarse groups}, Lecture Notes in Mathematics, \textbf{2339}, Springer, Cham, 2023; MR4696828

\bibitem[MRa23]{MRa23}
K. Mann and K. Rafi, Large-scale geometry of big mapping class groups, Geom. Topol. {\bf 27} (2023), no.~6, 2237--2296; MR4634747

\bibitem[MRo18]{MRo18}
K. Mann and C. Rosendal, Large-scale geometry of homeomorphism groups, Ergodic Theory Dynam. Systems {\bf 38} (2018), no.~7, 2748--2779; MR3846725

\bibitem[Mil68]{Mil68}
J.~W. Milnor, A note on curvature and fundamental group, J. Differential Geometry {\bf 2} (1968), 1--7; MR0232311

\bibitem[NY12]{NY12}
P.~W. Nowak and G.~L. Yu, {\it Large scale geometry}, EMS Textbooks in Mathematics, Eur. Math. Soc., Z\"urich, 2012; MR2986138

\bibitem[Roe03]{Roe03}
J. Roe, {\it Lectures on coarse geometry}, University Lecture Series, \textbf{31}, Amer. Math. Soc., Providence, RI, 2003; MR2007488

\bibitem[Ros22]{Ros22}
C. Rosendal, {\it Coarse geometry of topological groups}, Cambridge Tracts in Mathematics, 223, Cambridge Univ. Press, Cambridge, 2022; MR4327092

\bibitem[Ros23]{Ros23}
C. Rosendal, Geometries of topological groups, Bull. Amer. Math. Soc. (N.S.) {\bf 60} (2023), no.~4, 539--568; MR4642118

\bibitem[Tana]{Tang-mono}
R. Tang, {\it Categorical characterisations of quasi-isometric embeddings}, preprint, arXiv:2411.08501, 2024

\bibitem[Tanb]{Tang-rips}
R. Tang, {\it The metric Rips filtration, universal quasigeodesic cones, and hierarchically hyperbolic spaces}, preprint, arXiv:2511.16463, 2025

\bibitem[Zha15]{Zha15}
S. Zhang, Coarse quotient mappings between metric spaces, Israel J. Math. {\bf 207} (2015), no.~2, 961--979; MR3359724

\bibitem[Zha18]{Zha18}
S. Zhang, Asymptotic properties of Banach spaces and coarse quotient maps, Proc. Amer. Math. Soc. {\bf 146} (2018), no.~11, 4723--4734; MR3856140
  \end{thebibliography}
  \end{document}